\documentclass[11pt,b5paper,twoside, headrule]{amsart}

\usepackage{amsfonts, amsmath, amssymb,latexsym}
\usepackage{footnote}
\usepackage{epsfig}
\usepackage[curve]{xy}
\usepackage{algorithmic}
\usepackage{algorithm}
\usepackage{enumerate}
\usepackage{framed}
\usepackage{hyperref}
\usepackage{graphicx}

\footskip=10mm\parskip 0.2cm\addtocounter{page}{0}
\newtheorem{thm}{Theorem}[section]
\newtheorem{cor}[thm]{Corollary}
\newtheorem{lem}[thm]{Lemma}
\newtheorem{prop}[thm]{Proposition}

\theoremstyle{definition}
\newtheorem{defn}[thm]{Definition}

\newtheorem{example}[thm]{Example}

\theoremstyle{remark}
\newtheorem{rem}[thm]{Remark}
\numberwithin{equation}{section}

\newcommand{\Cc}{\mathcal{C}}

\newcommand{\cC}{\mathcal{C}}
\newcommand{\cP}{\mathcal{P}}
\newcommand{\cW}{\mathcal{W}}
\newcommand{\cM}{\mathcal{M}}
\newcommand{\cS}{\mathcal{S}}

\newcommand{\Z}{{\mathbb Z}}

\newcommand{\F}{{\mathbf F}}
\newcommand{\fq}{{\mathbb F}_q}
\newcommand{\fqc}{{\mathbb F}_{q^c}}
\newcommand{\fqt}{{\mathbb F}_{q^2}}

\newcommand{\Q}{{\mathbb Q}}
\newcommand{\rk}{{\rm rk}}
\newcommand{\ga}{{\alpha}}

\newcommand{\e}{{\epsilon}}

\newcommand\qbin[3]{\left[\begin{matrix} #1 \\ #2 \end{matrix} \right]_{#3}}

\begin{document}
	\renewcommand\baselinestretch{1.2}
	\renewcommand{\arraystretch}{1}
	\def\base{\baselineskip}
	\font\tenhtxt=eufm10 scaled \magstep0 \font\tenBbb=msbm10 scaled
	\magstep0 \font\tenrm=cmr10 scaled \magstep0 \font\tenbf=cmb10
	scaled \magstep0
	
	
	\def\evenhead{{\protect\centerline{\textsl{\large{I. Blanco-Chac\'on, E. Byrne, I. Duursma and J. Sheekey}}}\hfill}}
	
	\def\oddhead{{\protect\centerline{\textsl{\large{On the Zeta Function of a Rank Metric Code}}}\hfill}}
	
	\pagestyle{myheadings} \markboth{\evenhead}{\oddhead}
	
	\thispagestyle{empty}
	
	\author{I. Blanco-Chac\'on}
	\author{E. Byrne}
	\author{I. Duursma}
	\author{J. Sheekey}
	\address{I. Blanco-Chac\'on, E. Byrne, and J. Sheekey \vspace{-4mm} }
	\address{School of Mathematics and Statistics, University College Dublin}
	\address{I. Duursma \vspace{-4mm} }
	\address{Deptartment of Mathematics, University of Illinois Urbana-Champaign}

\title{Rank-Metric Codes and Zeta Functions}

\maketitle

\begin{abstract}
We define the rank-metric zeta function of a code as a generating function of its normalized $q$-binomial moments. We show 
that, as in the Hamming case, the zeta function gives a generating function for the weight enumerators of rank-metric codes. We further prove a functional equation and derive an upper bound for the minimum distance in terms of the reciprocal roots of the zeta function. Finally, we show invariance under suitable puncturing and shortening operators and study the distribution of zeroes of the zeta function for a family of codes.
\end{abstract}

\section{Introduction}

In \cite{D01,D99}, the zeta polynomial of a $q$-ary linear code of length $n$ and minimum distance $d$ is defined to be the unique polynomial $P_{\cC}(T)$ of degree at most $n-d$ such that the generating function has expansion around $T$ given by
\begin{equation}\label{eq:Duu} 
   \frac{P_{\cC}(T)}{(1-T)(1-qT)}(xT+y(1-T))^n = \cdots +\frac{A_{\cC}(x,y)-x^n}{q-1} T^{n-d}+ \cdots,
\end{equation}   
where $A_{\cC}(x,y)$ is the Hamming weight enumerator of $\cC$. For an MDS code $\cC$, the weight enumerator $A_{\cC}(x,y)$ is obtained by setting $P_\cC(T)=1$. In general the degree of $P_\cC(T)$ is given by $n+2-d-d^\perp \geq 0$ and is equal to zero if and only if the code is MDS. For codes with high minimum distance and high dual minimum distance, the weight enumerator is completely determined by the polynomial $P_\cC(T)$ of small degree. In such cases there is an advantage in encoding the weight enumerator $A_\cC(x,y)$ by the smaller polynomial $P_\cC(T)$. Among the properties of the zeta polynomial $P_\cC(T)$ is an easy transform between the polynomials $P_\cC(T)$ and $P_{\cC^\perp}(T)$ for the code and its dual, and invariance under puncturing and shortening for codes that have a transitive automorphism group. So that for example a cyclic code and its extended code use (\ref{eq:Duu}) with the same zeta polynomial $P_\Cc(T)$ but with different $n$ and $d$.
The special form of the generating function is motivated by its interpretation for Reed-Solomon codes. A Reed-Solomon codeword represents the values of a polynomial in the elements of a finite field. The generating function counts the number of zeroes of polynomials and therefore describes the weight distribution of Reed-Solomon codes.

In this paper we establish that the zeta polynomial and its properties have $q-$analogues for rank-metric codes. 
Rank-metric codes have a Singleton type upper bound for the parameters. Codes that attain the bound are called maximum rank distance codes (MRD codes). They include Gabidulin codes that are the $q$-analogues of Reed-Solomon codes. The rank distance distribution of a MRD code is uniquely determined by the parameters of the code. The rank-metric zeta function has the property that $P_\Cc(T) = 1$ for MRD codes, and thus in particular for Gabidulin codes. In general the zeta polynomial is of low degree if both the minimum distance and the dual minimum distance of a code are close to the Singleton bound.

In Section 2 we summarize some preliminary material and definitions.
In Section 3 we introduce the zeta function $Z_{\cC}(T)$ of a rank metric code $\cC$ in terms of its normalized $q$-binomial moments, and we define the $q-$analogue of the zeta polynomial for a rank-metric code with the generating function $Z_{\cC}(T)$. In analogy with the Hamming metric case $P_\cC(T)=1$ if and only if $\cC$ is an MRD code. 
In Section 4 we relate the weight enumerator of a rank-metric code $\cC$ with its zeta polynomial and obtain the rank-metric analogue of (\ref{eq:Duu}). Moreover, we show that coefficients of the zeta polynomial coincide with those in the representation of a weight enumerator as a $\Q$-linear combination of MRD weight enumerators. 
In Section 5 we establish relations between the polynomials $P_\cC(T)$ and $P_{\cC^\perp}(T)$ for a code and its dual. We then demonstrate that the zeroes of $P_\cC(T)$ can be related to the minimum distance of a code by giving an upper bound for the minimum distance in terms of these zeroes. 
In Section 6 we define operations of puncturing and shortening of rank-metric codes and show that the action of these on the {\em average} weight enumerator of a code can be realized in terms of $q$-derivatives. As in the Hamming metric case the zeta polynomial is invariant under puncturing and shortening. We introduce the normalized weight enumerator and express the action of puncturing and shortening on it in terms of $q$-commuting operators. 
In Section 7 we discuss the zeroes of the zeta polynomial. For a self-dual code, the reciprocal zeroes occur in pairs $\{ \alpha, q^m / \alpha \}$ and occur as conjugate pairs if and only if both have absolute value $q^{m/2}.$ We consider some classes of codes with zeta polynomials for which all complex zeroes have the same absolute value.
\section{Preliminaries}

We will assume that $m,n$ are positive integers with $n \leq m$ and that $q$ is a prime power. We write $\fq^{m \times n}$ to denote the $m \times n$ matrices with entries in $\fq$.
For any $X \in \fq^{m\times n}$, we write $\ker X$ or $X^\perp$ to denote the (right) nullspace of $X$ in $\fq^n$. That is, $$X^\perp :=\{y \in \fq^n : Xy^T = 0 \}.$$ 
Unless explicitly stated otherwise, we assume that $\cC$ is an $\fq$-linear subspace of $\fq^{m \times n}$.

\begin{defn}
	The {\em dual code} of $\cC$ is the $\fq$-linear code
	$$\cC^\perp:= \{Y \in \fq^{m \times n}: \mbox{Tr}(XY^t)=0 \mbox{ for all } X \in \cC\}.$$
\end{defn}

The map $(X,Y) \mapsto \mbox{Tr}(XY^t)$ defines an inner
product on the space $\fq^{m \times n}$, so we have $\dim(\cC^\perp)=mn-\dim(\cC)$ and $\cC^{\perp\perp}=\cC$.

\begin{defn}
	The \emph{rank distance} between matrices 
	$X,Y \in \fq^{m \times n}$ is $d(X,Y):=\mbox{rk}(X-Y)$.
	For 
	$|\cC| \ge 2$, the \emph{minimum rank 
		distance} of 
	$\cC$ is the integer defined by 
	$d(\cC):= \min\{d(X,Y) : X,Y \in \cC, \ X \neq Y\}$.
	The \emph{weight}
	\emph{distribution} of $\cC$ is the integer vector 
	$W(\cC)=(W_t(\cC) : 0 \le t \le n)$,
	where, for all $t \in \{0,...,n\}$, 
	$$W_t(\cC):=|\{X \in \cC : \mbox{rk}(X)=t\}|.$$       
	The rank metric weight enumerator of $\mathcal{C}$ is the bivariate polynomial
	\begin{equation}
	W_{\mathcal{C}}(x,y)=\sum_{t=0}^nW_{t}(\mathcal{C})x^{n-t}y^{t}.
	\label{weightenum}
	\end{equation}
\end{defn}

Recall that the Gaussian binomial or $q$-binomial coefficient is defined by 
$$\qbin{n}{r}{q} := \left\{ \begin{array}{ll}
                    \displaystyle{\frac{(q^n-1)(q^n-q)\cdots(q^n-q^{r-1})}{(q^r-1)(q^r-q)\cdots(q^r-q^{r-1})} }& \text{ if } r \in \{1,...,n\},\\
                    1                                                                           & \text{ if } r = 0, \\
                    0 & \text{ otherwise.}
                    \end{array}
                    \right.
                     $$
This quantity counts the number of $r$-dimensional subspaces of an $n$-dimensional subspace over $\fq$. Since $q$ is fixed throughout this paper, for brevity we write $\qbin{n}{r}{}$ to mean $\qbin{n}{r}{q}$. 
The rank-metric analogue of the Singleton bound says that $|\cC| \leq q^{m(n-d+1)}$ for any $\cC$ of minimum rank-distance $d$ \cite{D78}.
Rank metric codes that meet this bound are called maximum-rank-distance (MRD) codes. It has been known for some decades that such codes exist for all choices of $m,n,d$ \cite{D78,G85,R91}.
In \cite{D78} it was shown that the weight enumerator of an MRD code in $\fq^{m \times n}$ of minimum distance $d$ is uniquely determined and given by
$$M_{n,d}(x,y):= x^n+ \sum_{t=d}^n\sum_{i=d}^t(-1)^{t-i}q^{\binom{t-i}{2}}{{n}\brack{t}}{{t}\brack{i}}(q^{m(i-d+1)}-1)x^{n-t}y^t,$$
where the coefficient of $x^{n-t}y^t$ counts the number of matrices in the MRD code of rank $t$.
It is not hard to see that for fixed $n$ the $M_{n,d}$ are linearly independent over $\Q$. Any code $\cC$ whose dual code has minimum distance at least 2 has weight enumerator that can be expressed as a $\Q$-linear combination of the MRD weight enumerators. That is,
there exist $p_0,...,p_{n-d} \in \Q$ satisfying
$$ W_{\cC}(x,y) = p_0 M_{n,d}(x,y)+ \cdots + p_{n-d} M_{n,n}(x,y).$$
These coefficients will turn out to define the {\em zeta polynomial} of $\cC$, that we introduce in the next section.

Let $P$ be a partially ordered set (poset, from now on). The M\"obius function for $P$ is defined via the recursive formula
\begin{equation}
\begin{array}{l}
\mu(x,x)=1,\\
\displaystyle\mu(x,z)=-\sum_{x\leq y<z}\mu(x,y).
\end{array}
\end{equation}
\begin{lem}[M\"obius Inversion Formula]Let $f,g:P\to\mathbb{Z}$ be any two functions on $P$. Then
	\begin{itemize}
		\item[1. ]$\displaystyle f(x)=\sum_{x\leq y}g(y)\mbox{ if and only if }g(x)=\sum_{x\leq y}\mu(x,y)f(y)$.
		\item[2. ]$\displaystyle f(x)=\sum_{x\geq y}g(y)\mbox{ if and only if }g(x)=\sum_{x\geq y}\mu(y,x)f(y)$.
	\end{itemize}
\end{lem}
In particular, for the subspace lattice of $\mathbb{F}_q^{m\times n}$ (with partial order defined by set inclusion), and for two subspaces $U$ and $V$ of dimensions $u$ and $v$, we have that
\begin{equation}\label{eq:mobinv}
\mu\left(U,V\right)=\left\lbrace
\begin{array}{cl}
(-1)^{v-u}q^{\binom{v-u}{2}}&\mbox{ if }U\leq V\\
\\
0 & \mbox{ otherwise.}
\end{array}
\right.
\end{equation}

\section{The Zeta Function}

We introduce the {\em zeta function} of a rank metric code $\cC$ in terms of its normalized $q$-binomial moments, following the approach for the Hamming metric case as given in \cite{D04}. 
In order to do so we use the notion of a shortened code 
(cf. \cite{R16}).

\begin{defn}Let $U\subseteq \mathbb{F}_q\mathbb{F}_q^n$ be a subspace of dimension $u$. 
	The {\em shortened subcode} of $\mathcal{C}$ with respect to $U$ is
	$$
	\mathcal{C}_U:=\left\{X\in \mathcal{C}: U\leq X^\perp\right\}.
	$$
	The strict shortening of $\mathcal{C}$ by $U$ is
	$$
	\widehat{\mathcal{C}}_U:=\left\{X\in \mathcal{C}: U= X^\perp\right\}.
	$$
\end{defn}
Notice that $\mathcal{C}_U$ is a subspace (indeed a subcode of $\cC$), but $\widehat{\mathcal{C}}_U$ in general is not. Clearly every element of $\widehat{\cC}_{U}$
has rank exactly $n-u$.

\begin{defn}For $u\geq 0$, the $u$-th binomial moment of $\mathcal{C}$ is defined by
	$$
	B_u(\mathcal{C}):=\sum_{dim(U)=n-u}(|\cC_U|-1).
	$$
\end{defn}

Lemma \ref{lem:binmom} is a straightforward consequence of the following duality result.
The result is used in \cite{R16} to obtain a short proof of the MacWilliams' Identity for rank-metric codes.

\begin{lem}[\cite{R16}, Lemma 28]\label{lem:rav}
	Let $U$ be a subspace of $\fq^n$ of dimension $n-u$. Then
	$$|\cC_U| = \frac{|\cC||\cC^\perp_{U^\perp}|}{q^{m(n-u)}}.$$
\end{lem}
\begin{proof} Let $R_U$ be the subspace of $m \times n$ matrices over ${\mathbb F}_q$ with row vectors in~$U$, and let $b : \Cc \times R_U$ be the bilinear form with $b(X,Y) = \mbox{Tr}(XY^t).$ The left null space is $\Cc_U$ and the right null space is $R_U \cap C^\perp = C^\perp_{U^\perp}$. So  that $|\Cc|/|C_U| = |R_U|/|C^\perp_{U^\perp}|.$ 
\end{proof}

\begin{lem}\label{lem:binmom}
  Let $\cC$ have dimension $k$ and minimum rank distance $d$, and let $\cC^\perp$ have minimum distance $d^\perp$.
  Then 
  $$ B_u(\cC)= \left\{ \begin{array}{ll}
               0                               & \text{ if } u     < d \\
               (q^{k-m(n-u)}-1) \qbin{n}{u}{} & \text{ if } u     > n-d^\perp \\    
                       \end{array}  \right.$$
\end{lem}	

\begin{proof}
    Any $X \in \cC$ satisfies $\rk(X^\perp) \leq n-d$, so if $U$ has dimension $n-u > n-d$ then $\cC_U = \{0\}$.
    Similarly, $C^\perp_{U^\perp} = \{0\}$ if $\dim(U^\perp) = u > n - d^\perp$ since every element $X \in \cC^\perp$ has rank at least $d^\perp$. Then from Lemma \ref{lem:rav} we get that
     $$ |\cC_U|= \left\{ \begin{array}{ll}
    1                               & \text{ if } u     < d \\
    q^{k-m(n-u)}  & \text{ if } u     > n-d^\perp \\    
    \end{array}  \right..$$
    Since $\qbin{n}{u}{}$ counts number of subspaces of $\mathbb{F}_q^n$ of dimension $n-u$, the result follows. 
\end{proof}	

We remark that in the instance that $\cC$ is an MRD code, the values $B_u(\cC)$ of Lemma \ref{lem:binmom} were given in \cite{DGMS10}. 
Moreover, we may invoke Lemma \ref{lem:rav} along with the Singleton bound to deduce that if $U$ has dimension $u$ then
$$|\cC_U| \leq |C| / q^{mu} \leq q^{m(n-u-d+1)}.$$

\begin{defn}For $u\geq 0$, the $u$-th normalized binomial moment of $\mathcal{C}$ is defined by
	$$
	\displaystyle{b_u(\mathcal{C}):=\frac{B_{u+d}(\mathcal{C})}{{n \brack u+d}}}.
	$$
\end{defn}

Then $b_u(\mathcal{C})$ counts the average number of non-zero elements of the shortened code $\cC_U$ where $U$ has dimension $n-u-d$.

We extend the definition of $b_u$ to all $u \in \Z$, by setting
$$
b_u(\mathcal{C}) = \begin{cases}
0                &\text{ if } u < 0. \\
q^{k-m(n-u-d)}-1, &\text{ if } u > n-d^\perp-d.
\end{cases}
$$

Note that if $\cC$ is an MRD code of minimum distance $d$ then $k=m(n-d+1)$ and $d^\perp=n+2-d$.
Then $b_u(\cC) = q^{m(u+1)}-1$ for all $u\geq 0$. In particular, the values $b_u(\cC)$ are independent of the minimum distance of $\cC$.
For this reason we write $b_u$ instead of $b_u(\cC)$ in the cases that $\cC$ is MRD.

\begin{defn}\label{def:pu}
   The {\em zeta function} of $\cC$ (c.f. \cite{D04}) is defined by
   $$ Z_{\cC}(T) := (q^m-1)^{-1}\sum_{u \geq 0} b_u(\cC) T^u.$$	
\end{defn}

It is straightforward to check that for $u \notin \{0,...,n-d^\perp-d+2\}  $ we have the following relation among the $b_u$, namely,
\begin{equation}\label{eq:rel}
    b_u(\cC) - (q^m+1)b_{u-1}(\cC)  + q^m b_{u-2}(\cC) =0.
\end{equation}

\begin{defn}For $u\geq 0$, define
	$$
	p_u(\mathcal{C}):= (q^m-1)^{-1}(b_u(\cC)  - (q^m+1)b_{u-1}(\cC)  + q^m b_{u-2}(\cC) ).
	$$
	The {\em zeta polynomial} of $\cC$ is defined by 
	$$P_{\cC}(T):= \sum_{u=0}^{n-d^\perp-d+2} p_u(\cC) T^{u}.$$
\end{defn}
Clearly, for any $\cC$, the zeta polynomial $P_{\cC}(T)$ has degree at most $n-d+1$, and $p_u(\cC)=0$ for any $u > n-d-d^\perp +2$.
We will assume that the minimum distance of the dual code is at least 2, and write
$$P_{\cC} = \sum_{u=0}^{n-d} p_u(\cC) T^{u}. $$
The recursion relates the zeta function and zeta polynomial via
$$Z_{\cC}(T)  = (q^m+1)TZ_{\cC}(T) - q^m T^2 Z_{\cC}(T) + P_{\cC}(T), $$
and hence we obtain that the generating function $Z_{\cC}(T)$ satisfies the equation
$$Z_{\cC}(T) = \frac{P_{\cC}(T)}{(1-T)(1-q^mT)}. $$

This immediately yields the following for an MRD code.

\begin{lem}
	Let $\cC$ be an MRD code. Then $P_{\cC}(T)=1$ and 
	$$Z_{\cC}(T) = \frac{1}{(1-T)(1-q^mT)}. $$
\end{lem}

\begin{proof}
	$\cC$ is MRD if and only if $d+d^\perp = n+2$, in which case (\ref{eq:rel}) holds for all $u>0$ and $P_{\cC}(T)=1$.
\end{proof}

\section{Weight Enumerators and Zeta Functions}
We will establish a relation between the weight enumerator of a code and its zeta function, giving the rank metric analogue of \cite[Theorem 9.5]{D99}.
We first approach this using M\"{o}bius inversion.
Given a subspace $U$ of $\fq^n$, we have the relation:
\begin{equation}
\displaystyle |\cC_U|=\sum_{U\leq V} |\hat{\cC}_V|,
\label{gfromh}
\end{equation}
consequently, by the M\"obius inversion formula we get
\begin{equation}
 |\hat{\cC}_U|=\sum_{U\leq V}\mu(U,V) |{\cC}_V|,
\label{hfromg}
\end{equation}
for $\mu(U,V)$ defined as in (\ref{eq:mobinv}).

Next, we derive a more explicit description of $W_{\mathcal{C}}(x,y)$ in terms of the normalized binomial moments.

\begin{lem}Let $\mathcal{C}$ have minimum distance $d$ and let be $U$ a subspace of $\mathbb{F}_q^n$ of dimension $u$. Then
$$
 |\widehat{\cC}_U|=\sum_{w=0}^{n-d-u}(-1)^wq^{\binom{w}{2}}\sum_{\begin{array}{c}dim(V)=u+w\\U\leq V\end{array}}( |{\cC}_V|-1)
$$
\label{lemaaux1}
\end{lem}
\begin{proof}For each subspace $V$, denote its dimension by $v$. By the M\"obius inversion formula, we have
\begin{eqnarray*}\label{eq1}
 |\widehat{\cC}_U|& = &\sum_{U\leq V}(-1)^{v-u}q^{\binom{v-u}{2}}|{\cC}_V|,\\
                  & = &\sum_{U\leq V,v\leq n-d}(-1)^{v-u}q^{\binom{v-u}{2}}|{\cC}_V|  
                    + \sum_{U\leq V,v\geq n-d+1}(-1)^{v-u}q^{\binom{v-u}{2}}|{\cC}_V|\\
                  & = & \sum_{U\leq V,v\leq n-d}(-1)^{v-u}q^{\binom{v-u}{2}}|{\cC}_V|  
                  + \sum_{U\leq V,v\geq n-d+1}(-1)^{v-u}q^{\binom{v-u}{2}},\\  
                  & = & \sum_{v=u}^{n-d}(-1)^{v-u}q^{\binom{v-u}{2}}\sum_{U\leq V} |{\cC}_V|+\sum_{v=n-d+1}^{n}{{n-u}\brack{v-u}}(-1)^{v-u}q^{\binom{v-u}{2}}.
\end{eqnarray*}
using Lemma \ref{lem:binmom}.

Adding and subtracting $1$ to the inner sum of the left hand side, this becomes
\begin{equation}
\sum_{v=u}^{n-d}(-1)^{v-u}q^{\binom{v-u}{2}}\sum_{U\leq V}( |{\cC}_V|-1)+\sum_{v=u}^{n}{{n-u}\brack{v-u}}(-1)^{v-u}q^{\binom{v-u}{2}}.
\label{aux2lema1}
\end{equation}
Setting $v=u+w$, (\ref{aux2lema1}) equals
\begin{equation}
\sum_{w=0}^{n-d-u}(-1)^wq^{\binom{w}{2}}\sum_{U\leq V}( |{\cC}_V|-1)+\sum_{w=0}^{n-u}{{n-u}\brack{w}}(-1)^wq^{\binom{w}{2}}.
\label{aux3lema1}
\end{equation}
Since the right sum vanishes the result follows.
\end{proof}
\begin{prop}Let $\mathcal{C}$ have minimum distance $d$. Then
$$
W_{t}(\mathcal{C})=\sum_{i=d}^t(-1)^{t-i}q^{\binom{t-i}{2}}{n\brack t}{t\brack i}b_{i-d}(\mathcal{C}).
$$
\end{prop}
\begin{proof}
Let $t=n-u$,
and write $\displaystyle W_{n-t}(\mathcal{C})=\sum_{dim(U)=u} |\widehat{\cC}_U|$. Applying Lemma \ref{lemaaux1} and setting $u+w=v$, this equals
$$
\sum_{w=0}^{n-d-u}(-1)^wq^{\binom{w}{2}}\sum_{U\leq V}( |{\cC}_V|-1).
$$
For a fixed $V$, there are ${{u+w}\brack{u}}$ $U's$ contained in $V$ of dimension $u$. Hence
\begin{equation}
W_{t}(\mathcal{C})=\sum_{w=0}^{n-d-u}(-1)^wq^{\binom{w}{2}}{{u+w}\brack{u}}B_{n-(u+w)}(\mathcal{C}),
\end{equation}
which equals
$$
\sum_{w=0}^{t-d}(-1)^wq^{\binom{w}{2}}{{n}\brack{n-t+w}}{{n-t+w}\brack{n-t}}b_{t-w-d}(\mathcal{C})=\sum_{w=0}^{t-d}(-1)^wq^{\binom{w}{2}}{{n}\brack{t}}{{t}\brack{w}}b_{t-w-d}(\mathcal{C}).
$$
Setting $t-w=i$ and observing that ${{t}\brack{w}}={{t}\brack{t-w}}$, the result follows.
\end{proof}

Equivalently, the following statement holds. 
\begin{cor}Let $\mathcal{C}$ have minimum distance $d$. Then 
$$
W_{\mathcal{C}}(x,y)=x^n+\sum_{i=d}^n\left(\sum_{t=i}^n(-1)^{t-i}q^{\binom{t-i}{2}}{{n}\brack{t}}{{t}\brack{i}}x^{n-t}y^t\right)b_{i-d}(\mathcal{C}).
$$
\label{cor1}
\end{cor}

In particular, in the case of an MRD code with weight enumerator $M_{n,d}(x,y)$ we get

\begin{equation}\label{eq:mrd}
 M_{n,d}(x,y)-x^n = \sum_{i=d}^n\left(\sum_{t=i}^n(-1)^{t-i}q^{\binom{t-i}{2}}{{n}\brack{t}}{{t}\brack{i}}x^{n-t}y^t\right)
          (q^{m(i-d+1)}-1).          
\end{equation}

\begin{defn} For each $r \in \{0,...,n\}$ define
	$$\phi_{n,n-r}(x,y):=(q^m-1)^{-1}\left(M_{n,r}(x,y) -(q^m+1) M_{n,r+1}(x,y) +q^m M_{n,r+2}(x,y)\right),$$
	and $$\phi_n(T):= \sum_{r=0}^{n}\phi_{n,r}(x,y)T^r.$$
\end{defn}

Therefore,
\begin{equation}\label{eq:MRDzetawten}
     \frac{\phi_n(T)}{(1-T)(1-q^m T)} \equiv (q^m-1)^{-1} \sum_{r=0}^{n} \left(M_{n,r}(x,y)-x^n \right) T^{n-r} \pmod{T^{n+1}}
\end{equation}

\begin{lem}\label{lem:zetawt}
	For each $r \in \{0,...,d\}$, the coefficient of $T^{n-r}$ in the expression,
 $$Z_{\cC}(T) \phi_n(T)=\frac{P_{\cC}(T)\phi_n(T)}{(1-T)(1-q^m T)}$$
 is given by
 $$ (q^m-1)^{-1} \left(\sum_{i=0}^{n-r} p_i(\cC) M_{n,r+i}(x,y)-x^n \right). $$
 
\end{lem}

\begin{proof}
   From (\ref{eq:MRDzetawten}), the coefficient of $T^{n-r}$ in the left-hand-side of the above expression is
   $$ (q^m-1)^{-1}  \left(\sum_{i=0}^{n-r} p_i(\cC) M_{n,r+i}(x,y) -x^n\sum_{i=0}^{n-r} p_i(\cC) \right).$$
   Now from Definition \ref{def:pu} we have that
   $$\sum_{i=0}^{n-r} p_i(\cC)= \frac{b_{n-r}(\cC)-q^m b_{n-r-1}(\cC)}{q^m-1}.$$
   
   Then since $r\leq d$, we have $n-r -1> n-d^\perp-d$, and so $b_{n-r}(\cC)=q^{k-m(r-d)}-1$ and $b_{n-r-1}(\cC)=q^{k-m(r+1-d)}-1$. Therefore
$$\sum_{i=0}^{n-r} p_i(\cC)=\frac{(q^{k-m(r-d)}-1)-q^m(q^{k-m(r+1-d)}-1)}{q^m-1}= 1,$$
proving the claim.  	
\end{proof}

An explicit expression for $\phi_n(T)$ is given by:
\begin{lem}\label{lem:phi}
	$$
	\phi_n(T)=\sum_{i=0}^n\left(\sum_{t=i}^n(-1)^{t-i}q^{\binom{t-i}{2}}{{n}\brack{t}}{{t}\brack{i}}x^ty^{n-t}\right)T^{n-i}.
	$$
\end{lem}

\begin{proof}
	By definition, we have,
   	\begin{eqnarray*}
   	   (q^m-1)\phi_{n,n-r}(x,y)=M_{n,r}(x,y) -(q^m+1) M_{n,r+1}(x,y) +q^m M_{n,r+2}(x,y).
   \end{eqnarray*} 
    Then by (\ref{eq:mrd}), we get
    \begin{eqnarray*}
   	     & = & \sum_{t=r}^n(-1)^{t-r}q^{\binom{t-r}{2}}{{n}\brack{t}}{{t}\brack{r}}x^ty^{n-t}(q^m-1) \\
   	     &+&  \sum_{t=r+1}^n(-1)^{t-r-1}q^{\binom{t-r-1}{2}}{{n}\brack{t}}{{t}\brack{r+1}}x^ty^{n-t}(q^{2m}-1 -(q^m+1)(q^m-1))\\
   	     &+& \sum_{i=r+2}^n \sum_{t=i}^n(-1)^{t-i}q^{\binom{t-i}{2}}{{n}\brack{t}}{{t}\brack{i}}x^ty^{n-t}
   	     (b_{i-r+1}-(q^{m}+1)b_{i-r}+q^mb_{i-r-1})\\
   	     & = & \sum_{t=r}^n(-1)^{t-r}q^{\binom{t-r}{2}}{{n}\brack{t}}{{t}\brack{r}}x^ty^{n-t}(q^m-1),
   	\end{eqnarray*}
         as the third sum vanishes due to the relation (\ref{eq:rel}).
\end{proof}

Therefore, the weight enumerators $M_{n,d}(x,y)$ and $W_{\cC}(x,y)$ can be expressed as
\begin{eqnarray}
M_{n,d}(x,y) & = & x^n + \sum_{i=d}^n \phi_{n,n-i}(x,y)b_{i-d},\\
W_{\cC}(x,y) & = & x^n + \sum_{i=d}^n \phi_{n,n-i}(x,y)b_{i-d}(\cC).\label{eq:phizeta}
\end{eqnarray}

In fact, the polynomials $\phi_{n,n-r}(x,y)$ are related to a well known class of $q$-polynomials (see, e.g. \cite{GM04}), and are given by
$$\phi_{n,n-r}(x,y) = \qbin{n}{r}{}p_{n-r}(x,y)y^r, $$
where $p_k(x,y):=\prod_{j=0}^{k-1}(x-q^jy) = \sum_{j=0}^k (-1)^{k-j} q^{\binom{k-j}{2}} \qbin{k}{j}{} x^j y^{k-j}$.
Therefore,
\begin{eqnarray}
W_{\cC}(x,y) & = & x^n + \sum_{i=d}^n b_{i-d}(\cC)\qbin{n}{i}{}p_{n-i}(x,y)y^i.\label{eq:phizeta}
\end{eqnarray}

We now establish a connection between the weight enumerator of a code and its zeta function.

\begin{thm}
	The coefficient of $T^{n-d}$ in the expression,
	$$Z_{\cC}(T) \phi_n(T)=\frac{P_{\cC}(T)\phi_n(T)}{(1-T)(1-q^m T)}$$
	is given by
	$$ \frac{W_{\cC}(x,y)-x^n}{q^m-1} . $$
	In particular, the zeta polynomial $P_{\cC}(T)$ is the unique polynomial of degree at most $n-d$
	such that $\displaystyle{W_{\cC}(x,y) = \sum_{i=0}^{n-d} p_i M_{n,d+i}(x,y)}$.
\label{th49}
\end{thm}

\begin{proof}
	It is immediate from (\ref{eq:phizeta}) that the coefficient of $T^{n-d}$ in $Z_{\cC}(T) \phi_n(T)$ is
	$(q^m-1)^{-1}(W_{\cC}(x,y)-x^n)$. From Lemma \ref{lem:zetawt}, we have that this coefficient is equal to $(q^m-1)^{-1}(-x^n+ \sum_{i=0}^{n-d} p_i M_{n,d+i}(x,y))$, and thus 
	$\displaystyle{W_{\cC}(x,y) = \sum_{i=0}^{n-d} p_i(\cC) M_{n,d+i}(x,y)}$.
	Since the $M_{n,d+i}(x,y)$ are linearly independent over $\Q$, the result follows.
\end{proof}	

\begin{example}\label{exjohn}
Let us consider $m=n=3$. Then we have the following weight enumerators for MRD codes
\begin{align*}
M_{3,3} &= x^3+(q^3-1)y^3\\
M_{3,2} &= x^3+(q^3-1)(q^2+q+1)xy^2+(q^3-1)(q^3-q^2-q)y^3\\
M_{3,1} &= x^3+(q^3-1)(q^2+q+1)x^2y+(q^3-1)^2(q^2+q)xy^2+(q^3-1)(q^3-q)(q^3-q^2)y^3\\
\end{align*}
Taking $q=2$ we get
\begin{align*}
M_{3,3} &= x^3+7y^3\\
M_{3,2} &= x^3+49xy^2+14y^3\\
M_{3,1} &= x^3+49x^2y+294xy^2+168y^3\\
\end{align*}
We formally define $M_{3,4} = x^n$, the weight enumerator of the trivial code. Consider a code $\cC$ of {\it constant rank} $2$; that is, every nonzero codeword in $\cC$ has rank-weight two. If $\cC$ has dimension $k$, then $W_{\cC}(x,y) =x^3+(q^k-1)xy^2$. Such codes always exist for $k=3$, and exist with $k=4$, $q=2$. For $k=3$, we have a code $\cC_1$ withe weight enumerator
\[
W_{\cC_1}(x,y) = \frac{1}{q^2+q+1}\left(M_{3,2}-(q^3-q^2-q)M_{3,3}+q^3M_{3,4}\right),
\]
and hence
\[
P_{\cC_1}(T) = \frac{1-(q^3-q^2-q)T+q^3T^2}{q^2+q+1}.
\]
For $k=4$ and $q=2$, we get a code $\cC_2$ with zeta polynomial
\[
P_{\cC_2}(T) = \frac{1}{49}\left(15-30T+64T^2\right).
\]
This reflects the fact that
\[
x^3+15xy^2 = \frac{1}{49}\left(15(x^3+49xy^2+14y^3)-30(x^3+7y^3)+64x^4\right).
\]

From considering an MRD code of $3\times 3$ matrices with minimum distance $2$ over $\mathbb{F}_{q^2}$ as a subspace of $6\times 6$ matrices over $\mathbb{F}_q$, we produce a code $\cC_3$ with weight enumerator
\[
W_{\cC_3}(x,y) =x^6+(q^6-1)(q^4+q^2+1)x^2y^4+(q^6-1)(q^6-q^4-q^2)y^6,
\]
and zeta polynomial
\[
P_{\cC_3}(T) = \frac{1+ (-q^6 + q^4 + q^3 + q^2 + q)T + q^6T^2}{q^4 + q^3 + q^2 + q + 1}.
\]
Taking $q=2$ we get
\[
P_{\cC_3}(T) = \frac{1-34T + 64T^2}{31}.
\]

\end{example}

\section{Duality and an Upper Bound}
A number of authors have described the duality theory of rank metric codes \cite{D78,GY07,R16}.
In \cite{GY07}, the authors give a direct analogue of MacWilliams' duality theorem relating the weight enumerator of a code with that of its dual, or more precisely, with the $q$-transform of its dual.
Define an operation $\tau$ on homogeneous polynomials in two variables by
\begin{equation}
f^{\tau}(x,y):=\tilde{f}(x+(q^m-1)y,x-y),
\label{tau}
\end{equation}
where $\tilde{f}$ is the $q$-transform of $f$, which is a $\Q$-linear map. The following version of 
MacWilliams' duality theorem for rank metric codes can be read in \cite[Theorem~1]{GY07}. 
\begin{thm}\label{th:mw}
$$
W_{\mathcal{C}^{\perp}}(x,y)=\frac{1}{|\mathcal{C}|}W^{\tau}_{\mathcal{C}}(x,y).
$$
\label{macwilliams}
\end{thm}

If $W_{\mathcal{C}}=W_{\mathcal{C}^{\perp}}$, $\mathcal{C}$ is called formally self-dual.
Delsarte \cite{D78} showed that the dual code of an MRD code is also MRD; the dual of an MRD code with minimum distance $d$ is an MRD code with minimum distance $n-d+2$. Since an MRD code exists for all choices of $m,n,d$, 
the family of MRD weight enumerators is closed under $\tau$ and hence the MacWilliams' identity. 
More precisely, an immediate corollary of Theorem \ref{th:mw} and the previous observation yields that:
\begin{cor}
$$
M_{n,d}^{\tau}(x,y)=q^{m(n-d+1)}M_{n,n-d+2}(x,y).
$$
\label{taumcwilliams}
\end{cor}

\begin{cor}Let $\mathcal{C}\subseteq\mathbb{F}_q^{m\times n}$ be a matrix code with minimal rank distance $d$ and dimension $k$. Let $d^{\perp}$ be the minimal rank distance of the dual code $\mathcal{C}^{\perp}$. Setting
$$
W_{\mathcal{C}}(x,y)=\sum_{i=0}^rp_iM_{n,d+i}\mbox{ and }W_{\mathcal{C^{\perp}}}(x,y)=\sum_{j=0}^ts_jM_{n,d^{\perp}+j},
$$
we have that
\begin{itemize}
\item[a)] $r=t=n-d-d^{\perp}+2$;
\item[b)] $s_j=p_{r-j}q^{m(d^{\perp}+j-1)-k}$;
\end{itemize}
\end{cor}
\begin{proof}Applying MacWilliams' identity to $W_{\mathcal{C}}(x,y)$ and using Corollary \ref{taumcwilliams} gives
$$
W_{\mathcal{C}^{\perp}}(x,y)=\sum_{i=0}^rq^{m(n-d-i+1)-k}p_iM_{n,n-d-i+2}.
$$
Setting $d^{\perp}+j=n-d-i+2$, its minimal value ($j=0$) happens when $i=n-d-d^{\perp}+2$ and its maximal value ($i=0$) when $j=n-d-d^{\perp}+2$. Hence $a)$ follows. Claim $b)$ follows from the fact that the MRD family is a $\mathbb{Q}$-basis, hence the decomposition of $W_{\mathbb{C}^{\perp}}$ as a linear combination of MRD polynomials of different minimal distances must be unique.
\end{proof}

The MacWilliams' identity can be translated into a functional relation between the zeta function of a code and that of its dual 
(or just between the zeta function of the code at different arguments, if the code is formally self dual),
as we now show.

\begin{thm}
$$
Z_{\mathcal{C}^{\perp}}(T)=q^{m(n-d+1)+k}T^{r-2}Z_{\mathcal{C}}\left(\frac{1}{q^mT}\right).
$$
In particular, if $\mathcal{C}$ is formally self-dual, then
$$
Z_{\mathcal{C}}(T)=q^{m(n-d+1)+k}T^{r-2}Z_{\mathcal{C}}\left(\frac{1}{q^mT}\right).
$$
\label{th54}
\end{thm}
\begin{proof}Let $P_{\mathcal{C}}(T)=\displaystyle\sum_{i=0}^rp_iT^i$ and $P_{\mathcal{C}^{\perp}}(T)=\displaystyle\sum_{j=0}^rs_iT^i$, for some $p_i,s_i \in \Q$. 
By Corollary \ref{cor1}, for $0\leq j\leq r$:
$$
s_j=p_{r-j}q^{m(d^{\perp}+j-1)-k},
$$
hence multiplying by $T^{j-r}$, summing in $j$ and then changing $j$ to $r-j$:
\begin{eqnarray*}
T^{-r}P_{\mathcal{C}^{\perp}}(T)= \sum_{j=0}^rp_jq^{m(d^{\perp}+r-j-1)-k}T^{-j}=q^{m(d^{\perp}+r-1)-k}P_{\mathcal{C}}\left(\frac{1}{q^mT}\right).
\end{eqnarray*}
Thus
$$
P_{\mathcal{C}^{\perp}}(T)=T^rq^{m(n-d+1)-k}P_{\mathcal{C}}\left(\frac{1}{q^mT}\right).
$$

Taking the quotient with $(1-T)(1-q^mT)$, the result follows.
\end{proof}

\begin{example}
We consider the zeta polynomials of the duals of the codes from Example \ref{exjohn}. The code $\cC_2$ has $m=n=3$, $d=2$, $k=4$, $q=2$, and zeta polynomial
\[
P_{\cC_2}(T) = \frac{1}{49}\left(15-30T+64T^2\right).
\]
Thus
\begin{align*}
P_{\cC_2^{\perp}}(T) &=T^2 2^{3}P_{\cC}\left(\frac{1}{2^3T}\right)\\
&= \frac{1}{49}\left(4-15T+60T^2\right)
\end{align*}

The code $\cC_3$ has $m=n=6$, $d=4$, $k=12$, and zeta polynomial
\[
P_{\cC_3}(T) = \frac{1-34T + 64T^2}{31}.
\]
Thus
\begin{align*}
P_{\cC_3^{\perp}}(T) &=T^2 2^{6}P_{\cC_3}\left(\frac{1}{2^6T}\right) \\
&= \frac{1-34T + 64T^2}{31} = P_{\cC_3}(T).
\end{align*}
Note that although $\cC_3$ and $\cC_3^{\perp}$ have the same zeta polynomial (and hence the same zeta function), they do not have the same weight enumerator, as they have different minimum distances. In fact, if a code has quadratic zeta polynomial, which occurs if and only if $d+d^{\perp} = n$, then its dual will have the same zeta polynomial if and only if $k=m(n-d)$.
\end{example}

We close this section by showing how the zeta function can be used to derive an upper bound on the minimum distance of code, which is the rank-metric analogue of \cite[Section 2]{D01}.

\begin{thm}Let $P_{\mathcal{C}}(T)=p_0(1+aT+...)$ be the zeta polynomial of a rank metric code $\mathcal{C}\subseteq\mathbb{F}_q^{m\times n}$, of degree $r\leq n-d$ and let $a$ be the negative of the sum of its reciprocal roots. Then
$$
d\leq \log_q\left[(a+q^ m+1)(q-1)+1\right]-1.
$$
\end{thm}
\begin{proof}
First, observe that according to Theorem \ref{th49} the coefficient of $T^{n-d}$ in $Z_{\mathcal{C}}(T)\phi_n(T)$ is
\begin{equation}
\frac{W_{\mathcal{C}}(x,y)-x^ n}{q^m-1}.
\label{s5eq1}
\end{equation}
Second, since $\frac{1}{(1-T)(1-q^ mT)}$ is the common zeta-function of all the MRD weight enumerators (in particular, of those with length $n$), again due to Theorem \ref{th49}, we have
$$
\frac{\phi_n(T)}{(1-T)(1-q^ mT)}=\frac{1}{q^ m-1}\sum_{i=0}^{n-d}\left(M_{n,d+i}(x,y)-x^ n\right)T^{n-d-i}+...
$$
Likewise, the coefficient of $T^{n-d}$ in $Z_{\mathcal{C}}(T)\phi_n(T)$ is
\begin{equation}
\frac{1}{q^ m-1}\sum_{i=0}^{r}p_i\left(M_{n,d+i}(x,y)-x^ n\right).
\label{s5eq2}
\end{equation}
Within the coefficient of $T^ {n-d}$ in these two coincident expressions, we need to compare the coefficients of $x^{n-d}y^ d$ and $x^{n-d-1}y^{d+1}$.

The coefficient of $x^{n-d}y^d$ in (\ref{s5eq1}) is $\frac{W_{n-d}(\mathcal{C})}{q^m-1}$, $W_{n-d}(\mathcal{C})$ being the number of codewords of rank $d$ in $\mathcal{C}$, which equals the coefficient of $x^{n-d}y^d$ in (\ref{s5eq2}), i.e., in $\frac{p_0}{q^ m-1}\left(M_{n,d}(x,y)-x^n\right)$. Using Theorem \ref{th49}, this last expression can be expanded as
\begin{multline*}
\frac{p_0}{q^ m-1}\left(\sum_{i=d}^{n}\phi_{n,n-i}(x,y)b_{i-d}\right)= \\  = \frac{p_0}{q^ m-1}\sum_{i=d}^{n}{n\brack i}b_{i-d}y^ i\sum_{j=0}^{n-i}(-1)^{n-i-j}q^{{n-i-j\choose 2}}{n-i\brack j}x^jy^{n-i-j},
\end{multline*}
the coefficient of $x^{n-d}y^d$ in which expression is precisely $\frac{p_0}{q^m-1}{n\brack d}b_0$. Comparing with the corresponding coefficient in (\ref{s5eq1}) we obtain $p_0\geq 0$. Observe that the numbers $b_u$ are the normalized $q$-binomial coefficients of an MRD code of length $n$, and they are independent of the minimal distance $d$.

Now, we compare the coefficients of $x^{n-d-1}y^{d+1}$. First, this coefficient is $\frac{W_{n-d-1}(\mathcal{C})}{q^m-1}\geq 0$ in (\ref{eq1}). In (\ref{s5eq2}), this equals the coefficient of $x^{n-d-1}y^{d+1}$ in the sub-expression
$$
\frac{1}{q^m-1}\left(p_0(M_{n,d}(x,y)-x^n)+p_1(M_{n,d+1}(x,y)-x^n)\right).
$$
The coefficient in $\frac{p_1}{q^m-1}\left(M_{n,d+1}(x,y)-x^n\right)$ is, as before 
\begin{equation}
\frac{p_1}{q^m-1}{n\brack d+1}b_0.
\label{s5eq3}
\end{equation}
As for the coefficient in $\frac{p_0}{q^m-1}\left(M_{n,d}(x,y)-x^n\right)$, we expand this expression again as
$$
\frac{p_0}{q^m-1}\sum_{i=d}^{n}\phi_{n,n-i}(x,y)b_{i-d},
$$
from which we isolate the sought for coefficient, which is
\begin{equation}
\frac{p_0}{q^m-1}\left({n\brack d+1}b_1-{n\brack d}{n-d\brack n-d-1}b_0\right).
\label{s5eq4}
\end{equation}
Adding (\ref{s5eq3}) and (\ref{s5eq4}), dividing by $p_0$ and taking into account that ${n\brack d+1}=\frac{q^{n-d}-1}{q^{d+1}-1}{n\brack d}$, yields
\begin{equation}
ab_0\frac{q^{n-d}-1}{q^{d+1}-1}+b_1\frac{q^{n-d}-1}{q^{d+1}-1}-b_0{n-d\brack n-d-1}\geq 0.
\label{s5eq5}
\end{equation}
Since ${n-d\brack n-d-1}=\frac{q^{n-d}-1}{q-1}$, it holds that,
$$
ab_0+b_1-b_0\frac{q^{d+1}-1}{q-1}\geq 0.
$$
Since $b_u=q^{m(u+1)}-1$ for each $u\geq 0$, dividing by $b_0$ yields
$$
a+q^m+1\geq\frac{q^{d+1}-1}{q-1}.
$$
By taking logarithm the result follows.
\end{proof}

\section{Puncturing and Shortening}

Puncturing and shortening are fundamental operations of coding theory. In \cite{D01}, the zeta function was related to the {\em normalized weight enumerator} of a code, by successive puncturing and shortening of an MDS weight enumerator. We consider the rank-metric case in what follows.

We have already considered shortened subcodes in order to define the zeta function of a code. We will now consider puncturing and shortening as operations that yield codes in 
$\fq^{ m\times(n-1) }$, that is, as projections of codes in $\fq^{m \times n}$. Such operations have been considered in \cite{BR17}, with a slightly different definition. We will establish the invariance of the normalized weight enumerator of a code under puncturing and shortening.

\begin{defn}
	Let $H$ be a hyperplane in $\fq^n$. Fix a basis of $H$ and 
	let $P_H \in \fq^{ n\times(n-1) }$ be the matrix whose columns are the elements this basis of $H$, in some order.
	Let $h \in \fq^n \backslash H$.
	We define the punctured and shortened codes of $\cC$ with respect to $H$, respectively by:
	\begin{eqnarray*}
		\Pi_H(\cC) & := & \{  XP_H : X \in \cC\} \subset \fq^{ m\times(n-1) } \text{ (punctured code)},\\
		\Sigma_{h,H}(\cC)& := & \{ XP_H  : X \in \cC, Xh^T = 0 \}\subset \fq^{ m\times(n-1) } \text{ (shortened code)}.
	\end{eqnarray*}
\end{defn}

Clearly $\Pi_H(\cC)$ and $\Sigma_{h,H}(\cC)$ are only well-defined up to a choice of basis of $H$. We assert that this is sufficient for our purposes, which is to define operations of puncturing and shortening on {\em normalized} weight enumerators. 
With these definitions, $\cC$ can by punctured in any of $\qbin{n}{n-1}{}=\qbin{n}{1}{}$ ways and can be shortened in any of 
$q^{n-1}\qbin{n}{1}{}=|\{(\langle h \rangle,H ): \dim H = n-1, h \notin H\}|$.

\begin{lem}\label{lem:dist}
	Let $H$ be a hyperplane in $\fq^n$ and let $P_H \in \fq^{n \times (n-1)}$ have columns that form a basis of $H$. Let $X$ be non-zero in 
	$\fq^{m \times n}$. Then
	$$\rk(XP_H) = \left\{  \begin{array}{ll}
	                            \rk(X)   & \text{ if }X^\perp \not\subset H \\
	                            \rk(X)-1 & \text{ if }X^\perp \subset H  
	                       \end{array}
	                       \right.$$ 
\end{lem}	

\begin{proof}
	Let $h \in \fq^n \backslash H$. Then 
	$\rk(X) = \rk(X[P_H, h^T])=\rk([XP_H,Xh^T])=\rk(XP_H)$
	if and only if $Xh^T$ is contained in the column-space of $XP_H$, which holds if and only if $h \in X^\perp +H$.
	$X^\perp \not\subset H$ if and only if $\fq^n = X^\perp +H$, in which case $h \in X^\perp + H$ and $\rk(X)=\rk(XP_H)$.
	If $X^\perp \subset H$ then $h \notin X^\perp + H = H$, so $\rk(XP_H)=\rk(X)-1$. 
\end{proof}	

\begin{cor}
   Let $H$ be a hyperplane in $\fq^n$ and let $h \in \fq^n\backslash H$. 
   Let $\cC$ have parameters $[m \times n,k,d\geq 2]$ over $\fq$.
   Then the punctured code $\Pi_H(\cC)$ has parameters
   $[m \times (n-1),k,\geq d-1]$. The shortened code $\Sigma_{h,H}(\cC)$ has parameters $[m \times (n-1),\geq k-m,\geq d]$.
\end{cor}
\begin{proof}
	The lower bounds on the minimum distances of the codes follows directly from Lemma \ref{lem:dist}.
	Let $P_H \in \fq^{n \times (n-1)}$ have columns that form a basis of $H$. Then $XP_H = 0$ if and only if $H \subset X^\perp$, in which case $X=0$, since otherwise $\dim X^\perp =n-\rk(X)\leq n-d \leq n-2$. It follows that $\cC$ and $\Pi_H(\cC)$ have the same cardinality.
	
	Let $h \in \fq^n \backslash H$. Consider the map $\theta_h :\fq^{m \times n} \longrightarrow \fq^m: X \mapsto Xh^T$.
	Then $\cC \cap \ker \theta_h=\cC_{\langle h \rangle}:=\{ X \in \cC : Xh^T = 0 \}$. 
	Therefore $\dim \cC_{\langle h \rangle} = \dim (\cC \cap \ker \theta_h) \geq k - m$. The result now follows since $\Sigma_{h,H}(\cC)$ is obtained as the punctured code of $\cC_{\langle h \rangle}$ with respect to $H$, which both have the same dimension. 
	\end{proof}	
Given an arbitrary code, $\cC$ and hyperplanes $H_1,H_2$, the punctured codes $\Pi_{H_1}(\cC)$ and $\Pi_{H_2}(\cC)$ may have different weight enumerators. Neither is the weight enumerator of a shortened code $\Sigma_{h,H}(\cC)$ necessarily uniquely determined.
However, in the case of an MRD code $\cC$ we observe:

\begin{cor}\label{cor:mrd}
Let $d\geq 2$. The shortened and punctured codes of an MRD code with weight enumerator $M_{n,d}(x,y)$ are also MRD and have weight enumerators $M_{n-1,d}(x,y)$ and $M_{n-1,d-1}(x,y)$, respectively.
\end{cor}
\begin{proof}
	Let $H$ be a hyperplane in $\fq^n$. 
    Applying Corollary \ref{cor:mrd} together with the rank-metric Singleton bound, we see that the punctured code $\Pi_H(\cC)$ has parameters $[m \times (n-1),m(n-d+1),d-1]$ and is thus MRD.
    Similarly the shortened code $\Sigma_{h,H}(\cC)$ has parameters $[m \times (n-1),m(n-d),d]$ and is also an MRD code.  
\end{proof}	

We wish to define algebraic operations on the weight enumerator of a code corresponding to the expected weight enumerator of its punctured or shortened codes. In the Hamming metric case, such operations are defined in terms of partial derivatives \cite{D01}. In the rank-metric case,
we may implement $q$-analogues of partial derivatives (c.f. \cite{GM04}, \cite{K97}). We define such operations as follows.

\begin{defn}
	Let $f(x,y) = \sum_{i=0}^n f_i x^{n-i}y^i \in \Q[x,y]$. The partial $q$-derivatives of $f(x,y)$ with respect to $x$ and $y$ are defined by:
	\begin{eqnarray*}
		D_{x}(f(x,y)) & := & \frac{f(qx,y)-f(x,y)}{(q-1)x} =  \sum_{i=0}^n f_i \qbin{n-i}{1}{} x^{n-i-1}y^{i},\\
		D_{y}(f(x,y)) & := & \frac{f(x,qy)-f(x,y)}{(q-1)y} =  \sum_{i=0}^n f_i \qbin{i}{1}{} x^{n-i}y^{i-1},\\
		D_{q,x}(f(x,y)) & := & \frac{f(qx,qy)-f(x,qy)}{(q-1)x} =  \sum_{i=0}^n f_i q^{i}\qbin{n-i}{1}{} x^{n-i-1}y^{i}.		
	\end{eqnarray*}
\end{defn}
It is straightforward to check that these are linear operators.
In analogy with \cite[Section 3]{D01} we now define the operations of puncturing and shortening on the homogeneous polynomials of degree $n$ in $\Q[x,y]$ by

$$ \cP:=\qbin{n}{1}{}^{-1}\left( D_{q,x}+ D_{y}\right) \;\text{ and } \cS:=\qbin{n}{1}{}^{-1}	D_{x} .$$
Let us consider the {\em average} weight enumerator arrived at over all possible shortenings or puncturings (cf. \cite{D01}).

\begin{thm}
   The average weight enumerators over all shortened and punctured codes of $\cC$ are given by 
   \begin{enumerate}
   	\item $\displaystyle{\cP (W_{\cC}(x,y))=\qbin{n}{1}{}^{-1}\sum_{\dim H = n-1} W_{\Pi_H(\cC)}(x,y)},$
   	\item $\displaystyle{\cS (W_{\cC}(x,y))=\qbin{n}{1}{}^{-1}\frac{1}{q^{n-1}}\sum_{ \dim H = n-1,\langle h \rangle \not\subset H} W_{\Sigma_{h,H}(\cC)}(x,y)}.$	
   \end{enumerate}	
\end{thm}
\begin{proof}
Let $H $ be a hyperplane in $\fq^n$ and let $P_H \in \fq^{n \times (n-1)}$ have columns that form a basis of $H$.
From Lemma \ref{lem:dist}, $\rk(XP_h)=\rk(X)$ if and only if $X^\perp \not \subset H$. 
Let $X\in \cC$ have rank $w>0$. Then $\rk(XP_H)=w-1$ if and only if $X^\perp \subset H$ 

Since there are $\qbin{n-(n-w)}{n-1-(n-w)}{}=\qbin{w}{1}{}$ hyperplanes in $\fq^n$ containing $X^\perp$,
$X$ corresponds to a word of rank $w-1$ in $\qbin{w}{1}{}$ punctured codes of $\cC$ and to a word of rank $w$ in $\qbin{n}{1}{}-\qbin{w}{1}{} = q^{w} \qbin{n-w}{1}{}$ punctured codes of $\cC$. 
Therefore, each term $x^{n-w}y^w$ of the weight enumerator of $\cC$ yields the contribution 
$$  q^{w} \qbin{n-w}{1}{}x^{n-w-1}y^w +  \qbin{w}{1}{}  x^{n-w}y^{w-1}, $$ in the sum of the $\qbin{n}{1}{}$ different weight enumerators of all possible punctured codes.

For any $h \notin H$, $XP_H \in \Sigma_{h,H}(\cC)$ if and only if $h \in X^\perp$. There are $|X^\perp \backslash X^\perp \cap H|$ such vectors $h$ and so
$q^{n-w-1}$ one dimensional subspaces $\langle h \rangle$ such that $XP_H \in \Sigma_{h,H}(\cC)$.
As outlined above, there are $q^{w} \qbin{n-w}{1}{}$ hyperplanes not containing $X^\perp$.
In particular 
each term $x^{n-w}y^w$ of the weight enumerator of $\cC$ yields the contribution
$$ q^{n-w-1} q^{w} \qbin{n-w}{1}{} x^{n-1-w}y^w = q^{n-1}\qbin{n-w}{1}{},$$ 
in the sum of the $q^{n-1}\qbin{n}{1}{}$ different weight enumerators of all possible shortened codes. 
\end{proof}

\begin{example}\label{exeim}
	Consider the ${\mathbb F}_2$-$[3 \times 3, 4
	,2]$ code 
	$$\cC =\left\langle \begin{bmatrix}
	1 & 0 & 0 \\ 0 & 0 & 0 \\ 0 & 1 & 0
	\end{bmatrix},
	\begin{bmatrix}
	0 & 0 & 1 \\ 1 & 0 & 0 \\ 0 & 0 & 0
	\end{bmatrix}, 
	\begin{bmatrix}
	0 & 1 & 0 \\ 0 & 0 & 1 \\ 0 & 0 & 0
	\end{bmatrix}, 
	\begin{bmatrix}
	0 & 0 & 1 \\ 0 & 1 & 0 \\ 0 & 1 & 0
	\end{bmatrix}\right\rangle,
	$$
	which has weight enumerator $W_{\cC}(x,y) = x^3 + 13 x y^2 + 2 y^3.$
	Then 6 hyperplanes $H$ yield shortened codes with distribution of weight enumerators similar to those for $H=001^\perp$,
	and the hyperplane $H=010^\perp$ yields 4 shortened codes $\Sigma_{h,H}$ of order 2.
	\begin{center}
	\begin{tabular}{ccc}
		\hline 
		$H$ & $h$ & $W_{\Sigma_{h,H}}(x,y)$ \\
		\hline
		$001^\perp$ & $001$ & $x^2 + 3y^2$ \\
		            & $111$ & $x^2 + y^2$ \\
		            & $011$ & $x^2 + y^2$ \\
		            & $101$ & $x^2 + 3y^2$ \\
	\end{tabular}
    \begin{tabular}{ccc}
    	\hline 
    	$H$ & $h$ & $W_{\Sigma_{h,H}}(x,y)$ \\
    	\hline
    	$010^\perp$ & $110$ & $x^2 + y^2$ \\
    	            & $010$ & $x^2 + y^2$ \\
    	            & $111$ & $x^2 + y^2$ \\
    	            & $011$ & $x^2 + y^2$ \\
    \end{tabular}
    \end{center}
    Then the average weight enumerator over all shortened codes is given by
    \begin{eqnarray*}
       & & \qbin{3}{1}{}^{-1}(1/4)\sum_{ \dim H = 2,\langle h \rangle \not\subset H} W_{\Sigma_{h,H}(\cC)}(x,y)\\
       & =& 
          (1/28)(28x^2 +52y^2) = x^2+13/7 y^2\\
       & = & \cS(x^3 + 13 x y^2 + 2 y^3) = \qbin{3}{1}{}^{-1}\left(\qbin{3}{1}{} x^2 + 13y^2\right).    
    \end{eqnarray*}
\end{example}
\begin{lem}
	The zeta polynomial $P_\cC(T)$ is invariant under the shortening and puncturing operations $\cS$ and $\cP$.
\end{lem}
\begin{proof}
From Corollary \ref{cor:mrd}, we have 
$$\cS(W_{\cC}(x,y)) = \cS\left(\sum_{i=0}^{n-d} p_i(\cC) M_{n,d+i}(x,y) \right) = \sum_{i=0}^{n-1-d} p_i(\cC) M_{n-1,d+i} (x,y),$$
and 
$$\cP(W_{\cC}(x,y)) = \cP\left(\sum_{i=0}^{n-d} p_i(\cC) M_{n,d+i}(x,y) \right) = \sum_{i=0}^{n-d} p_i(\cC) M_{n-1,d+i-1} (x,y).$$
\end{proof}
\begin{defn}
	The normalized weight enumerator of $\cC$ is defined to be the polynomial,
	$$\cW_{\cC}(T):=  (q^m-1)^{-1}\sum_{i=d}^n \qbin{n}{i}{}^{-1} W_{i}(\cC)T^{i-d} .$$
\end{defn} 

We write $\cW_{\cC}^{\cP}(T)$ and $\cW_{\cC}^{\cS}(T)$ to denote the normalized weight enumerators corresponding to $\cP(W_\cC(x,y))$ and $\cS(W_\cC(x,y))$, respectively.

\begin{thm}\label{th:nwe} Let $\cC$ have minimum distance $d$. Then,	
	\begin{enumerate}
		\item $\cW_{\cC}^{\cP}(T) =q^d T \cW_{\cC}(qT) + \cW_{\cC}(T) -(q^m-1)^{-1} W_n(\cC)q^n T^{n-d+1},$\label{th:nwe1} 
		\item $\cW_{\cC}^{\cS}(T) =\cW_{\cC}(T) - (q^m-1)^{-1}W_n(\cC) T^{n-d} .$ 
	\end{enumerate}
\end{thm}
\begin{proof}
	To show (1), we apply the operation $\cP=\qbin{n}{1}{}^{-1}\left( D_{q,x}+ D_{y}\right)$ to $W_{\cC}(x,y)$.
	\begin{eqnarray*}
		\cP(W_\cC(x,y)) &~=~& \qbin{n}{1}{}^{-1}\left(\qbin{n}{1}{} x^{n-1} + \sum_{i=d}^{n-1} W_{i}(\cC) q^i \qbin{n-i}{1}{} x^{n-i-1} y^{i} 
		+ \sum_{i=d}^{n}   W_{i} (\cC)\qbin{i}{1}{} x^{n-i} y^{i-1}\right)\\
		&~=~& x^{n-1} + \qbin{n}{1}{}^{-1}\left(\sum_{i=d}^{n-1} W_{i}(\cC)  q^i \qbin{n-i}{1}{} x^{n-i-1} y^{i}
		+ \sum_{i=d}^{n}   W_{i} (\cC) \qbin{i}{1}{}x^{n-i} y^{i-1}\right) \\  
		&~=~&    x^{n-1}  +\qbin{n}{1}{}^{-1} \sum_{i=d-1}^{n-1} \left(  W_{i}(\cC)  q^i \qbin{n-i}{1}{}  + W_{i+1} (\cC) \qbin{i+1}{1}{} \right)x^{n-1-i}y^i . 
	\end{eqnarray*}
	Therefore, the associated normalized polynomial $\cW_{\cC}^{\cP}(T)$ is given by:
	\begin{eqnarray*}
		 &~=~&  (q^m-1)^{-1}\qbin{n}{1}{}^{-1} \sum_{i=d-1}^{n-1} \left(  W_{i}(\cC)  q^i \qbin{n-i}{1}{}  + W_{i+1} (\cC) \qbin{i+1}{1}{} \right) \qbin{n-1}{i}{}^{-1} T^{i-d+1} \\
		&~=~&  (q^m-1)^{-1}\sum_{i=d-1}^{n-1} \left(  W_{i}(\cC)  q^{i} \qbin{n}{i}{}^{-1}  + W_{i+1} (\cC) \qbin{n}{i+1}{} ^{-1}\right)  T^{i-d+1} \\
		&~=~&  (q^m-1)^{-1}\sum_{i=d}^{n}  W_{i}(\cC)  q^{i} \qbin{n}{i}{}^{-1}T^{i-d+1} -(q^m-1)^{-1}q^nW_n(\cC) T^{n-d+1} + \\
		&  & \quad +~ (q^m-1)^{-1}\sum_{i=d}^{n} W_{i} (\cC) \qbin{n}{i}{} ^{-1}  T^{i-d},\\
		&~=~&  q^d T \cW_{\cC}(qT) +\cW_{\cC}(T) -(q^m-1)^{-1}q^nW_n(\cC) T^{n-d+1}.
	\end{eqnarray*}
	The proof of (2) is similar.
\end{proof}

Consider the following operations $\ga,\e$ on rational functions $f(T)$ in indeterminate $T$.
\begin{eqnarray}
     \ga f(T):=Tf(T)= \text{ and } \e f(T):= f(qT).
\end{eqnarray}

Now $\ga,\e$ form a $q$-commuting pair \cite{K97}, and obey the relations
\begin{eqnarray}\label{eq:qprod} \e \ga = q\ga \e,\; q\ga = \ga q,\; q \e= \e q, \end{eqnarray} 
with respect to composition. 
Then $\ga,\e$ are non-commuting operators with respect to the $q$-product determined by the relations in (\ref{eq:qprod})
and generate the $\Q$-algebra $\langle \ga,\e \rangle$, which acts on the space of rational functions.
We thus express Theorem \ref{th:nwe} (1) as
\begin{eqnarray}\label{eq:aleppunc}
\cW_{\cC}^{\cP}(T) \equiv (1 + q^d \ga \e)\cW_{\cC}(T) =  (1 + q^{d-1} \e \ga )\cW_{\cC}(T) \mod T^{n-d+1}, 
\end{eqnarray} 
An immediate corollary of Theorem \ref{th:nwe} is as follows.
\begin{cor}
	Let $\cC$ have minimum distance $d$. For $0\leq i\leq d$ we have that,
	\begin{enumerate}
		\item $(1+q^d \ga \e)^i\cW_{\cC}^{\cS}(T) \equiv (1+q^d \ga \e)^i\cW_{\cC}(T)  \mod T^{n-d} $, 
		\item $(1+q^d \ga \e)^i\cW_{\cC}^{\cP}(T) \equiv (1+q^d \ga \e)^{i+1}\cW_{\cC}(T)  \mod T^{n-d+1} $.
	\end{enumerate}
	In particular, for $d_\cS=d,d_\cP=d-1$ and $n_\cP = n_\cS= n-1$ we have
	\begin{enumerate}
		\item $(1+q^d \ga \e)^{d_\cS}\cW_{\cC}^{\cS}(T)  \equiv (1+q^d \ga \e)^{d}\cW_{\cC}(T)  \mod T^{n_\cS-d_\cS+1} $,
		\item $(1+q^d \ga \e)^{d_\cP}\cW_{\cC}^{\cP}(T)  \equiv (1+q^d \ga \e)^{d}\cW_{\cC}(T)  \mod T^{n_\cP-d_\cP+1}.$
	\end{enumerate}
\end{cor} 
We write $\cW_{\cC}^{\cP^r}(T)$ to denote the normalized weight enumerator that results from applying the puncturing operation
$r$ times to $\cW_{\cC}(T)$.
Note that the operators $(1+q^r\ga\e)$ and $(1+q^s\ga\e)$ commute, so there is no ambiguity in the following statement.
\begin{cor}\label{cor:puncs}
	Let $\cC$ have minimum distance $d$.
	Then 
	\begin{eqnarray*}
	\cW_{\cC}^{\cP^r}(T) & \equiv & \prod_{j=0}^{r-1}(1+q^{d-j} \ga \e)\cW_{\cC}(T) \mod T^{n-d+1}\\
	                     & =      & \sum_{j=0}^r q^{j(d-r+j)} \qbin{r}{j}{}T^j \cW_{\cC}(q^j T). 
	\end{eqnarray*}
\end{cor}
\begin{proof}
	The first equation is derived by repeated applications of (\ref{eq:aleppunc}). The well known identity
	$\prod_{j=0}^{r-1}(1+q^jy) = \sum_{j=0}^{r}q^{\binom{j}{2}} \qbin{r}{j}{}y^j$, along with the fact that 
	$(\ga \e)^j = q^{\binom{j}{2}}\ga^j\e^j$, yields the equation
	$$\prod_{j=0}^{r-1}(1+q^j(q^{d-r+1} \ga \e)) = \sum_{j=0}^{r}q^{j(d-r+j)} \qbin{r}{j}{} \ga^j \e^j.$$
	Applying this to $\cW_{\cC}(T)$ results in the 2nd equation.
\end{proof}	

\begin{example}\label{ex:mrd2ws}
	Let $\cM_{m\times n,d}(T)$ be the normalized weight enumerator of an ${\mathbb F}_2$-$[m \times n,m(n-d+1),d]$ code. Then successively puncturing
	$\cM_{7\times 7,4}(T)$, we arrive at the normalized weight enumerator of the whole space $\cW_{\F_2^{7 \times 4}}(T) = \cM_{7\times 4,1}(T)$.
	\begin{eqnarray*}
	    \cM_{7\times 7,4}(T)& = & 1 + 98 T + 9688 T^2 + 610112 T^3 \\
	    \cM_{7\times 6,3}(T)& = & 1 + 114 T + 12824 T^2 + 1230144 T^3\\
	                        &\equiv& (1+2^4\ga \e) \cM_{7\times 7,4}(T) \mod T^4\\
	                        & = & 1 + 114 T + 12824 T^2 + 1230144 T^3 + 78094336 T^4 \\ 
	   \cM_{7\times 5,2}(T) & = & 1 + 122 T + 14648 T^2 + 1640512 T^3\\
	                        &\equiv& (1+2^3\ga \e) \cM_{7\times 6,3}(T) \mod T^4\\
	                        & = & 1 + 122 T + 14648 T^2 + 1640512 T^3 + 78729216 T^4 \\
	                        &\equiv& (1+2^3\ga \e) (1+2^4\ga \e) \cM_{7\times 7,4}(T) \mod T^4\\
	                        & = & 1 + 122 T + 14648 T^2 + 1640512 T^3 + 156823552 T^4 + 9996075008 T^5 \\ 
	\cW_{\F_2^{7 \times 4}}(T) & = & 1 + 126 T + 15624 T^2 + 1874880 T^3 \\
	                           &\equiv& (1+2^2\ga \e) \cM_{7\times 5,2}(T) \mod T^4\\ 
	                           & = & 1 + 126 T + 15624 T^2 + 1874880 T^3 + 52496384 T^4 \\
	                           &\equiv& (1+2^2\ga \e) (1+2^3\ga \e) \cM_{7\times 6,3}(T) \mod T^4\\
	                           &\equiv& (1+2^2\ga \e)(1+2^3\ga \e) (1+2^4\ga \e) \cM_{7\times 7,4}(T) \mod T^4\\ 
	                           & = & 1 + 126 T + 15624 T^2 + 1874880 T^3 \\
	                           & + & 209319936 T^4 + 20032782336 T^5 + 1279497601024 T^6.            
	\end{eqnarray*}
\end{example}

If we consider the action of $\ga$ and $\e$ on $\Q[T]/\langle T^s\rangle$ for some positive integer $s$, then they may be represented as $s \times s$ rational matrices:
$$
\ga=\left( \begin{array}{ccccc}
0 & 0       & \cdots & 0      & 0      \\
1 &      0  & \cdots & 0      & 0     \\
\vdots & \ddots  & \vdots & \vdots & \vdots \\
0 &  \cdots  & 1   & 0      & 0 \\
0 &  \cdots   & 0 & 1      & 0 \\
\end{array}  \right) \text{ and }
\e=\text{diag}(1,q,q^2,...,q^{s-1}),
$$
so that $\langle \ga, \e \rangle$ form a sub-algebra of $\Q^{s \times s}$.
Then as an element of $\Q^{s \times s}$, $\ga$ is nilpotent so the inverse of  the operator $(1+q^t \e \ga)$ exists and is given by 
$$(1+q^t \e \ga)^{-1} = 1-q^t \e \ga +(q^t \e \ga)^2+\cdots +(-1)^{s-1}(q^t \e \ga)^{s-1}.$$ 
For example, with $s=3$, we have
$$1 + q^t \e \ga = 
\left( \begin{array}{ccc}
1       & 0       & 0 \\
q^{t+1} & 1       & 0 \\
0       & q^{t+2} & 1
\end{array} \right) \text{ and }
(1 + q^t \e \ga)^{-1} = 
\left( \begin{array}{ccc}
1       & 0       & 0 \\
-q^{t+1} & 1       & 0 \\
q^{2t+3}       & -q^{t+2} & 1
\end{array} \right) 
$$

Recall the following formulae (see \cite{GM04}): 
$$(T;q)_{\ell}: = \prod_{j=0}^{\ell-1}(1-q^j T) \text{ and }(T;q)^{-1}_{\ell}=\sum_{j=0}^{\infty}\qbin{\ell+j-1}{j}{}T^j.$$
Then $\prod_{j=0}^{r-1}(1+q^j(q^{d-r+1}\ga \e)) = (-q^{d-r+1}\ga \e;q)_r$ and so
$$(-q^{d-r+1}\ga \e;q)_r^{-1} = \sum_{j=0}^{\infty} (-1)^j\qbin{r+j-1}{j}{}q^{j(d-r+1)+\binom{j}{2}} \ga^j\e^j. $$

\begin{lem}\label{lem:puncs}
	Let $\cC$ have minimum distance $d$.
	Then 
	\begin{eqnarray*}
		\cW_{\cC}(T) & \equiv & (-q^{d-r+1}\ga \e;q)_r^{-1} \cW_{\cC}^{\cP^r}(T)  \mod T^{n-d+1}\\
		             & =      & \sum_{j=0}^\infty (-1)^jq^{j(d-r+1)+\binom{j}{2}} \qbin{r+j-1}{j}{}T^j \cW_{\cC}^{\cP^r}(q^j T). 
	\end{eqnarray*}
\end{lem}

In particular, puncturing is an invertible operation on normalized weight enumerators, modulo $T^{n-d+1}$.

\begin{example}

As $\prod_{j=0}^{i-1}(q^m-q^j)\qbin{\ell}{i}{}$ is the number of $m \times \ell $ matrices of rank $i$, we see that 
$$\cW_{\fq^{m \times \ell}}(T) = \sum_{i=1}^{\ell} \prod_{j=1}^{i-1}(q^m-q^j) T^{i-1}.$$
As in Example \ref{ex:mrd2ws}, the full space $\fq^{m \times \ell}$ is an MRD code with parameters $[m \times \ell,m\ell,1]$ and can be obtained by successive puncturings of an MRD code. 
Let $\cM(T)$ be the normalized weight enumerator of an MRD \\$[m \times (n+r),m(n-d+1),d+r ]$ code for some non-negative integer $r$.
Then the previous observations along with Corollary \ref{cor:puncs} implies that
	\begin{eqnarray*}
	\cM(T) &\equiv& \sum_{i=1}^{n-d+1} \prod_{j=1}^{i-1}(q^m-q^j) \prod_{j=0}^{d+r-1}(1+q^{d+r-j}\ga \e)^{-1} T^{i-1} \mod T^{n-d+1}  \\
	       &=     & (q^m-1)^{-1}\sum_{i=1}^{n-d+1} q^{im}(q^{-m};q)_i (-q^{d-r+1} \ga \e;q)_r^{-1} T^{i-1}\\
	       & =    & (q^m-1)^{-1}\sum_{i=1}^{n-d+1} q^{im}(q^{-m};q)_i
	                \sum_{j=0}^\infty (-1)^{j}q^{j(d-r+i)+\binom{j}{2}} \qbin{r+j-1}{j}{}T^{i+j-1}.
	\end{eqnarray*}

\end{example}
\begin{lem}\label{lem:nwef}
	Let $\cC$ have minimum distance $d$. Then the normalized weight enumerator of the code $\cC$ satisfies
	$$\cW_{\cC}(T)  \equiv (q^m-1)^{-1}\sum_{i=0}^{n-d} W_{d+i}\qbin{n}{i+d}{}^{-1} \frac{T^{i}}{(T;q)_{i+1}} \mod T^{n-d+1}.$$
\end{lem}

\begin{proof}
	We have 
	$$\frac{T^{i}}{(T;q)_{i+1}} = \sum_{j=0}^{\infty}\qbin{i+j}{j}{}T^{j+i} = \sum_{j=i}^{\infty}\qbin{j}{i}{}T^{j},$$
	which yields
	\begin{eqnarray*}
	\sum_{i=0}^{n-d} W_{d+i}\qbin{n}{i+d}{}^{-1} \frac{T^{i}}{(T;q)_{i+1}}
	=\sum_{i=0}^{n-d} W_{d+i}\qbin{n}{i+d}{}^{-1} \sum_{j=i}^{\infty}\qbin{j}{i}{}T^{j} \\
	= \sum_{j=0}^{\infty}T^{j} \sum_{i=0}^{n-d} W_{d+i}\qbin{n}{i+d}{}^{-1}\qbin{j}{i}{} = 
	  \sum_{j=0}^{\infty}T^{j}  W_{d+j}\qbin{n}{j+d}{}^{-1} 
	\end{eqnarray*}
\end{proof}

\section{Zeroes of the Zeta Polynomial}

For a self-dual code $\Cc$, Theorem \ref{th54} shows that the (in general complex) zeroes of the zeta polynomial $P_\Cc(T)$ occur in pairs
$( \alpha, 1/(q^m \alpha) )$. The two zeroes in a pair are of the same absolute value if and only if $|\alpha| = (q^m)^{-1/2}.$ 
Writing $\zeta(s) = Z(T=(q^m)^{-s}),$ this is the case if and only if as zeroes of $\zeta(s)$ they satisfy $\mathrm{Re}\,s = 1/2$. The latter is the critical line for the classical Riemann zeta function as well as for the zeta function of a curve over a finite field, which we recall here for convenience of the reader.

If $C$ is a non-singular projective curve defined over $\mathbb{F}_q$, for $k\geq 1$, denote by $N_k$ the number of $\mathbb{F}_{q^k}$-rational points of $C$, i.e., the cardinality of $C(\mathbb{F}_{q^ k})$. The zeta-function of $C$ is the formal power series

$$
Z(C,T)=exp\left(\sum_{k\geq 1}\frac{N_k}{k}T^k\right).
$$

This expression is well defined as a formal power series (cf. \cite[Appendix C]{H77}) and satisfies the following properties:
\begin{thm}[Weil, Dwork]The zeta function of any non-singular projective curve of genus $g$ can be expressed as
$$
Z(C,T)=\frac{P(T)}{(1-T)(1-qT)},
$$
with $P(T)\in\mathbb{Z}[T]$ a polynomial of degree $2g$, called the zeta-polynomial of $C$. Moreover, for each root $\omega_i$ of $P(T)$ ($1\leq i\leq 2g$), we have
$$
|\omega|=q^{1/2}.
$$
\end{thm}

As we see, $Z(C,T)$ is primarily used as a generating function for the number of points on a curve. There is no immediate analogue of this property for linear codes. In another interpretation, the zeta function of a curve describes the growth rate for the dimensions of linear systems on the curve. This is analogous to describing the growth rate of the binomial moments of a linear code. For codes for which the growth rate of the binomial moments is close to the growth rate of dimensions of linear systems the zeroes of the zeta polynomial will lie on the critical line. For the Hamming metric, this occurs for remarkably many codes (more so when the field size is small), including for infinite families of extremal weight enumerators. The zeta polynomial defined in this paper measures the growth rate of binomial moments for rank-metric codes, and we may ask whether this growth rate is such that we can expect to find codes with zeroes of the zeta polynomial on the critical line. 
We give an example where the zeta polynomial is an exact match to the zeta polynomial of a (Hasse-Weil maximal) elliptic curve. 

\begin{example}
	To construct a formally self-dual $4 \times 4$ rank-metric code we take the extended binary QR-code of length 18 in double-circulant form. We puncture at the last coordinate in one circulant block and shorten at the last coordinate in the other block. We then read out the remaining 16 coordinates in order and from each group of four coordinates form a row for a codeword in $4 \times 4$ format. The rank distribution for the 256 codewords is $0^1, 2^{21}, 3^{162}, 4^{72}.$ The normalized binomial moments are $1,1,8/5,16,256$ and the zeta polynomial is $(1+8T+16T^2)/25=(1+4T)^2/25$. The zeroes $T=-1/4$ (twice) have absolute value $(q^m)^{-1/2}=1/\sqrt{16}$ and thus lie on the critical line. The zeta polynomial is that of a maximal elliptic curve over the field of 16 elements. 
\end{example}

\begin{rem}
Note that constructing a self-dual rank-metric code from a self-dual linear code in this way is not unique; in other words, equivalent codes of length 16 may lead to inequivalent $4\times 4$ rank-metric codes. The rank-weight distributions obtained need not bear any relations to each other. It requires further research as to what distributions, and hence what zeta polynomials, may occur from codes constructed in this way.
\end{rem}

Heuristics for the Hamming metric are that a sufficient condition for a matching growth rate, and thus for zeroes on the critical line, are that a (formally self-dual) code has large enough minimum distance and weight distribution close to that of a random code. Precise formulations of these observations and their verification are an open problem.

We applied the same construction as in the example to double-circulant codes of length 38 to obtain formally self-dual $6 \times 6$ rank-metric codes. Among those, 12 codes had minimum rank distance 3 and thus a quadratic zeta polynomial. The rank distributions for these codes are not far apart and in all 12 cases the zeroes of the zeta polynomial are on the critical line. Among codes with lower minimum rank distance a wider range of rank distributions occurs. The distributions close to the average have zeroes on the critical line and those farther away have pairs of real zeroes. Typically one of the real zeroes among $( \alpha, 1/(q^m \alpha) )$ is close to 1 and the other close to $1/q^m$. This agrees with observations for the Hamming metric. These preliminary observations suggest that the zeroes of zeta polynomials for the rank-metric have properties similar to those for the Hamming metric. It remains to make these properties precise and to establish to what extent or under which additional assumptions the properties are similar to those for curves.

\subsection{Self-dual divisible codes} Here we consider a family of codes which are divisible (that is, the rank-weight of every codeword is divisible by some integer $c>1$) and also formally self-dual. It is known that for codes in the Hamming metric, such codes can only exist for $(q,c)\in \{(2,2),(2,4),(3,3),(4,2)\}$. In \cite{D01} the zeta polynomials of self-dual divisible codes that are also {\it random} were considered. 

 Here we show the existence of formally self-dual divisible codes for $c=2$ and all values of $q$. However we note that this construction never leads to random codes.

The matrix space $M_{m\times n}(\fqc)$ can be naturally embedded in the space $M_{cm\times cn}(\fq)$. The rank of every matrix in the image of this map is a multiple of $c$; in fact, we have 
$$
\mbox{rk}_{\fq}(A)=c\cdot \mbox{rk}_{\fqc}(A).
$$
Thus from codes over an extension field, we can produce {\it divisible codes} easily. This is in contrast to the Hamming metric case, where the Hamming weight of a vector does not behave well when reducing the field of interest.

A code $\Cc'$ of dimension $k$ over $\fqc$ produces a code $\Cc$ of dimension $kc$ over $\fq$. If we wish $\Cc$ to be formally self-dual, we must have $k= mnc/2$. As $k\leq mn$, we need $c=2$, and thus $\Cc'$ is in fact the full space $M_{m\times n}(\fqt)$.

For the image of the space $M_{2\times 2}(\fqt)$ in $M_{4\times 4}(\fq)$, we get the zeta polynomial
\[
P(T)=\frac{1+ (-q^4 + q^2 + q)T +q^4T^2}{q^2+q+1}.
\]
For example taking $q=2$ we get the zeta polynomial $P(T)=(1-10T+16T^2)/7$. The discriminant of this polynomial is positive for all values of $q\geq 2$, and so we get two real roots, whose product is $q^{-4}$.

For the image of the space $M_{3\times 3}(\fqt)$ in $M_{6\times 6}(\fq)$, we get the zeta polynomial
{\footnotesize
\[
P(T)=\frac{1+(-q^6 + q^2 + q)T+ (-q^8 - q^7 + q^6 + q^4 + q^3)T^2 + (-q^{12} + q^8 + q^7)T^3+q^{12}T^4}{q^4+q^3+q^2+q+1}
\]
}
For example taking $q=2$ we get $P(T) = (1-58T-296T^2-3712T^3+4096T^4)/15$, which has two real roots and one pair of conjugate complex roots.

In general, the image $\Cc$ of the space $M_{m\times m}(\fqt)$ in $M_{2m \times 2m}(\fq)$ is a formally self-dual code with minimum rank distance $d=2$ and all rank-weights divisible by $c = 2$. Let $\rho=q^{-m}.$  The zeta polynomial $P(T)$ for $\Cc$ has two real roots. As $m \to \infty$, the roots converge to $1$ and $\rho^2$. The rescaled polynomial $p(T) = P(\rho T) / P(0) = 1 + \cdots + T^{2m-2}$ is self-reciprocal. It has two real roots that converge to $\rho^{-1}$ and $\rho$ as $m \to \infty$ and $2m-4$ complex roots that lie on the unit circle. As $m \to \infty$, the factor $p_1(T)$ of $p(T)$ with two real roots converges to $1-(\rho^{-1}+\rho)T+T^2.$
The cofactor $p_2(T)$ with complex roots converges to $1+T^{2m-4}$. 

The coefficients of $p(T)$ approach $0$ as $m \to \infty$ except for the coefficients at $1, T, T^2$ and their reciprocals. The coefficients of $p(T)$ at $1, T$ and $T^2$ can be obtained with (\ref{eq:mrd}) by expressing the rank-weight enumerator of $\Cc$ as a linear combination of the $\fq$ rank-weight enumerators $M_{2m,d}$, for $d \geq 2$.
We find   
\begin{multline*}
\lim _{m \to \infty} p(T) = 1 - (q^m-\rho(q^2+q)) T - (q^2+q-\rho^2(q^6+q^4+q^2) T^2 + \cdots + \\ - (q^2+q-\rho^2(q^6+q^4+q^2) T^{2m-4} - 
 (q^m-\rho(q^2+q)) T^{2m-3} + T^{2m-2}.
\end{multline*} 
As $m \to \infty$, and thus $\rho = q^{-m} \to 0$, $p(\rho) \to 0.$ Thus $p_1(T) = 1-(\rho^{-1}+\rho)T+T^2$ divides the limit
and the cofactor $p_2(T)$ satisfies
\begin{multline*}
\lim_{m \to \infty} p_2(T) =  
1 + (q^2+q+1)\rho T + (q^6+q^4+2q^2+q+1) \rho^2 T^2 + \cdots  + \\ +  (q^6+q^4+2q^2+q+1) \rho^2 T^{2m-6} +  (q^2+q+1)\rho T^{2m-5} + T^{2m-4}.
\end{multline*}
The limits converge fast. We find that the complex zeros of the zeta polynomial for the image of $M_{m\times m}(\fqt)$ in $M_{2m \times 2m}(\fq)$
all have the same absolute value $\rho$ and, already for small $m$, are close to the zeros of $1+(q^mT)^{2m-4}=0$.

\begin{center}
\begin{figure}
 \vspace{-3cm}
 \includegraphics[width=10cm]{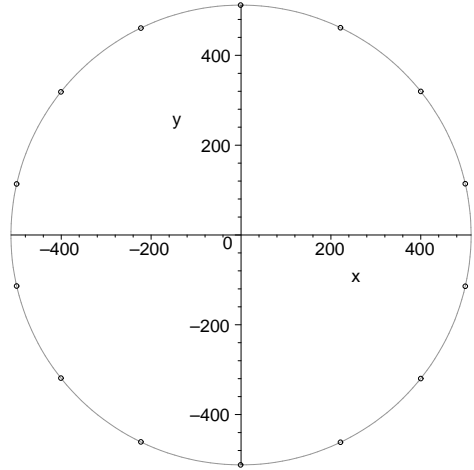}
 \vspace{-3cm}
 \caption{Complex zeroes for the image of $M_9(\mathbb{F}_4)$ in $M_{18}(\mathbb{F}_2)$.} 
\end{figure}
\end{center}

\end{document}